\documentclass[12pt]{article}
\usepackage{amsmath, amsfonts, amssymb, amsthm}
\usepackage{epsfig}

\usepackage[ps2pdf,linktocpage,bookmarks,bookmarksnumbered,bookmarksopen]{hyperref}
\hypersetup{%
   pdftitle   ={On the Small Deviation Problem for Some Iterated Processes},
   pdfsubject ={Manuscript},
   pdfauthor  ={Frank Aurzada and Mikhail Lifshits},
   pdfkeywords={Small deviations; small ball problem; iterated Brownian motion; iterated fractional Brownian motion; iterated process; local time.}
      }

\let\BFseries\bfseries\def\bfseries{\BFseries\mathversion{bold}} 
\newcommand{\ind}{1\hspace{-0.098cm}\mathrm{l}}
\newcommand{\eps}{\varepsilon}
\theoremstyle{plain}
\newtheorem{thm}{Theorem}
\newtheorem{lem}[thm]{Lemma}
\newtheorem{prop}[thm]{Proposition}
\newtheorem{cor}[thm]{Corollary}
\theoremstyle{definition}

\newtheorem{rem}[thm]{Remark}

\renewenvironment{proof}[1][] {\smallskip \noindent {\bf Proof#1:} }{\hspace*{\fill}$\square$\medskip\par}
\def\P{{\bf {\mathbb{P}}}}
\newcommand{\pr}[1]{\P\left(#1\right)}

\def\E{\mathbb{E}}

\def\R{\mathbb{R}}\def\Z{\mathbb{Z}} \def\d{\mathrm{d}}

\newcommand{\esssup}{\operatorname*{esssup}}
\newcommand{\deq}{\stackrel{d}{=}}

\begin{document}
\title{On the Small Deviation Problem \\for Some Iterated Processes}
\author{Frank Aurzada and Mikhail Lifshits}
\date{June 16, 2008}
\maketitle
\begin{abstract} 
We derive general results on the small deviation behavior 
for some classes  of iterated processes. This allows us, in particular, 
to calculate the rate of the small deviations for $n$-iterated Brownian motions and, more generally, 
for the iteration of $n$ fractional Brownian motions. 
We also give a new and correct proof  of some results in \cite{nane}. 
\end{abstract}

\bigskip


\section{Introduction}

This article is concerned with the small deviation problem for iterated processes. 
We consider two independent, real-valued stochastic processes $X$ and $Y$ 
(precise assumptions are given below), define the iterated process by 
$(X\circ Y)(t):=X(Y(t))$, $t\in [0,1]$, and investigate the small deviation function 
\begin{equation}
-\log \pr{ \sup_{t\in[0,1]} |X(Y(t))| \leq \eps},
\label{e:sdit} 
\end{equation}
when $\eps\to 0$.

The goal of this article is
\begin{itemize}
 \item to provide general results concerning the order of (\ref{e:sdit}) 
   -- given that we know the small deviation probabilities for 
   the processes $X$ and $Y$, respectively, and that $Y$ has a {\it continuous modification};
 \item to study some nice examples of processes to which this technique can be 
 applied, among them the iteration of $n$ (fractional) Brownian motions; 
    and
 \item to show how the technique can be modified if $Y$ has jumps. This is illustrated 
    by several examples, among them the $\alpha$-time Brownian motion, previously studied 
    in \cite{nane}. Here, we give a correct proof of (a weaker version of) the results 
    from \cite{nane}.
\end{itemize}

Small deviation problems (also called small ball problems or lower tail probability problems) 
were studied intensively during recent years, which is due to many connections to other subjects 
such as the functional law of the iterated logarithm of Chung type, strong limit laws in statistics, 
metric entropy properties of linear operators, quantization, and several other approximation quantities 
for stochastic processes. For a detailed account, we refer to the surveys \cite{lif} and 
\cite{lishao} and to the literature 
compilation \cite{sdbib}.

The interest in iterated processes, in particular iterated Brownian motion, 
started with the works of Burdzy (cf.\ \cite{burdzy1} and \cite{burdzy2}). 
Iterated processes have interesting connections to higher order PDEs, cf.\ 
\cite{alloubazehng} and \cite{nourdin} for some recent results.
Small deviations of iterated processes or the corresponding result for the law of the iterated 
logarithm are treated in \cite{huplvshi} ($X$ and $Y$ Brownian motions), \cite{khoshlewis} 
($X$ Brownian motion, $Y=|Y'|$ with $Y'$ being Brownian motion), \cite{nane} (see Section~\ref{sec:alphatime} below), 
\cite{lindeshi} ($X$ fractional Brownian motion, $Y$ a subordinator), and, most recently, 
\cite{lindezipfel} ($X$ fractional Brownian motion, $Y$ a subordinator, and the sup-norm is taken over 
a possibly fractal index set). 

In Section~\ref{sec:general}, we give general results under the assumption that the small 
deviation probabilities of $X$ and $Y$, respectively, are known to some extent and that $Y$ has a continuous 
modification. The proofs for these results are given in Section~\ref{sec:proofs} and the results 
are illustrated with several examples in Section~\ref{sec:examples}. In Section~\ref{sec:alphatime}, 
we treat examples where $Y$ has jumps, in particular, the so-called $\alpha$-time Brownian motion, 
studied earlier in \cite{nane}.
Finally, we mention some possible extensions and applications of our results and collect some open questions in 
Section~\ref{sec:remarks}.

\section{General results} \label{sec:general}
Before we formulate our main results, let us define some notation. 
We write $f\preceq g$ or
$g\succeq f$ if $\limsup f/g <\infty$, while the equivalence $f\approx g$ means that
we have both $f\preceq g$ and $g\preceq f$. Moreover, $f\lesssim g$ or
$g\gtrsim f$ say that $\limsup f/g \leq 1$. Finally, the strong equivalence $f\sim g$ means that
$\lim f/g=1$.

We say that a process $X$ 
is $H$-self-similar if $(X(ct))\deq (c^H X(t))$ for all $c>0$, where $\deq$ means that the 
finite-dimensional distributions coincide. Recall that, for example, fractional Brownian 
motion with Hurst parameter $H$ is $H$-self-similar. However, there are many interesting 
self-similar processes outside the Gaussian framework, e.g.\ a strictly 
$\alpha$-stable L\'{e}vy process is $1/\alpha$-self-similar (\cite{st}, \cite{embrechtsmaejima}, \cite{samorodnitsky}).

Let us consider stochastic processes $(X(t))_{t\geq 0}$ and $(Y(t))_{t\geq 0}$ that are 
independent and such that $X(0)=0$ and $Y(0)=0$ almost surely. We extend $X$ for $t<0$ in 
the usual manner using an independent copy: namely, let $X'$ be an independent copy of $X$, 
and set $X(t):=X'(-t)$ for all $t<0$. We call this process {\it two-sided}.

Recall that if $X$ is a classical fractional Brownian motion, it has dependent 
``wings'' $(X(t))_{t\geq 0}$ and $(X(t))_{t\leq 0}$; hence it does not fit in the
scope of the present section. Nevertheless we will show below how the technique can be adjusted
by using the stationarity of increments instead of the independence. 

In this section, we assume that 
\begin{itemize}
\item $X$ is an $H$-self-similar, two-sided process. 
\item $Y$ has a continuous modification. 
\end{itemize}

If we know the weak asymptotic order of the small deviation probability of the processes 
$X$ and $Y$, respectively, we can determine that of the process $X\circ Y$.

\begin{thm} 
Let $\theta,\tau>0$. Then, under the above assumptions, the relations
\begin{eqnarray} - \log \pr{ \sup_{t\in[0,1]} |X(t)| \leq \eps} 
&\approx & \eps^{-\theta} \notag \\ 
- \log \pr{ \sup_{t\in[0,1]} |Y(t)| \leq \eps} 
&\approx & \eps^{-\tau} \label{e:weak:y}
\end{eqnarray}
imply 
$$-\log \pr{ \sup_{t\in[0,1]} |X(Y(t))| \leq \eps} \approx  \eps^{-1/(1/\theta+H/\tau)}.
$$
The implication also holds if $\approx$ is replaced by $\preceq$ or $\succeq$, respectively. 
For translating lower bounds (i.e.\ $\preceq$ in the relations above), the assumption that 
$Y$ is continuous can be dropped. 
\label{thm:main}
\end{thm}

\begin{rem} 
Note that the resulting exponent is always less than $\theta$. Therefore, 
the small deviation probability of $X\circ Y$ is always larger than the one of $X$. 
This is of course not true in general when comparing $X\circ Y$ to $Y$. 
\end{rem}

\begin{rem} \label{rem:noselfsim} 
In fact, for the proof it is sufficient to know that 
$$
-\log \pr{ \sup_{t\in [0,T]} |X(t)| \leq \eps} 
\approx T^{\theta H} \eps^{-\theta},\qquad\text{when $\eps\to 0$,}
$$ 
for all $T>0$, instead of the self-similarity property and the given small deviations of $X$. 
\end{rem}

Furthermore, provided we know the {\it strong} order of the small deviation functions, we can prove 
a result for the strong asymptotic order for that of the iterated process.

\begin{thm} 
Let $\tau>0$ and $\theta:=1/H>0$. Then, under the above assumptions, the relations
\begin{eqnarray} 
- \log \pr{ \sup_{t\in[0,1]} |X(t)| \leq \eps} &\sim & k \eps^{-\theta}\notag 
\\ 
- \log \pr{ \sup_{s,t\in[0,1]} |Y(s)-Y(t)| \leq \eps} 
  &\sim & \kappa \eps^{-\tau} 
\label{e:knownrange}
\end{eqnarray}
imply 
$$
-\log \pr{ \sup_{t\in[0,1]} |X(Y(t))| \leq \eps} 
\sim \kappa^{1/(1+\tau)} \tau^{-\tau/(1+\tau)} (1+\tau) k^{\tau/(1+\tau)}   
\eps^{-\theta \tau/(1+\tau)}.
$$
The implication also holds if $\sim$ is replaced by $\lesssim$ or $\gtrsim$, respectively.  
For translating lower bounds (i.e.\ $\lesssim$ in the relations above), 
the assumption that $Y$ is continuous can be dropped. 
\label{thm:mainstrong}
\end{thm}

It is easy to check that this theorem recovers the results from \cite{huplvshi}, where $X$ and $Y$ are Brownian motions, and \cite{khoshlewis}, where $X$ is a Brownian motion and $Y=|Y'|$ with $Y'$ being a Brownian motion.

\begin{rem} 
As in Remark~\ref{rem:noselfsim}, it is sufficient to know that 
$$
-\log \pr{ \sup_{t\in [0,T]} |X(t)| \leq \eps} 
\sim k T^{\theta H} \eps^{-\theta},\qquad\text{when $\eps\to 0$,}
$$ for all $T>0$, instead of the self-similarity property and the given 
small deviations of $X$.
\end{rem}

\begin{rem}
A careful reader would wonder why the self-similarity index of $X$ and the small deviation
index of $X$ should be related by $\theta:=1/H$.  In fact, this relation is rather
typical {\it for the supremum norm}. We refer to \cite{LS} for the explanation of this
fact in the context of small deviations of general norms. Also see \cite{samorodnitsky}.
\end{rem}

One may argue that, typically, not the probability in (\ref{e:knownrange}) is given but 
rather the small deviation probability. The following lemma translates from the small 
deviation probability into (\ref{e:knownrange}) (and backwards) if we know that a process 
satisfies the {\it Anderson property}.

Recall that the Anderson property for a random vector $Y$ taking values in a linear space $E$ means
that
\begin{equation} \label{anderson}
\pr{ Y \in A} \ge \pr{ Y \in A+e}
\end{equation}
for any $e\in E$ and any measurable symmetric convex set $A\subseteq E$, cf.\ \cite{anderson}. It is known that any centered
Gaussian vector has this property. Another example is given by symmetric $\alpha$-stable vectors
since their distributions can be represented as mixtures of Gaussian ones.

\begin{lem} 
Let $(Y(t))_{t\in T}$ be a stochastic process with $Y(t_0)=0$ a.s.\ for some 
$t_0\in T$. Furthermore, assume that $Y$ satisfies the Anderson property. 
Let $\tau>0$ and $\ell$ be a slowly varying function. Then we have
$$
- \log \pr{ \sup_{s,t\in T} |Y(s)-Y(t)| \leq \eps} \sim  \kappa \eps^{-\tau} \ell(\eps)
$$
if and only if 
$$
- \log \pr{ \sup_{t\in T} |Y(s)| \leq \eps} \sim  \kappa 2^{-\tau} \eps^{-\tau} \ell(\eps).
$$ 
\label{lem:anders} 
\end{lem}

Note that the applicability of Lemma~\ref{lem:anders} depends on the use of the Anderson property. 
We now state that if $X$ satisfies the Anderson property so does $X\circ Y$. This makes 
it possible to use Theorem~\ref{thm:mainstrong} iteratively.

\begin{lem} \label{lem:anderson} 
Let $T$ be some non-empty index set and let $(X(u))_{u\in\R}$ and $(Y(t))_{t\in T}$ be 
independent stochastic processes, where $X$ satisfies the Anderson property. Then the process 
$(X(Y(t)))_{t\in T}$ satisfies the Anderson property. 
\end{lem}

This shows that, in particular, iterated Brownian motion, the iteration of two (or more general $n$) 
fractional Brownian motions, $\alpha$-time Brownian motion (defined below), and many other non-Gaussian
processes satisfy the Anderson property.

\section{Proofs of the general results} \label{sec:proofs}
Before we prove Theorem~\ref{thm:main}, we recall a result that translates the small 
deviation probability into a corresponding result for the Laplace transform.

\begin{lem} 
Let $Y(0)=0$ almost surely, $p>0$, and $\tau>0$. Then 
$$
-\log \pr{ \sup_{t\in[0,1]} |Y(t)| \leq \eps} \approx \eps^{-\tau},\qquad \eps\to 0,
$$ 
implies 
$$
-\log \E \exp \left( - \lambda \sup_{s,t\in[0,1]} |Y(t)-Y(s)|^p \right) 
\approx \lambda^{1/(1+p/\tau)},\qquad \lambda\to \infty.
$$ 
The relation also holds if $\approx$ is replaced by $\preceq$ ($\succeq$ in the assertion) or 
$\succeq$ ($\preceq$ in the assertion), respectively. 
\label{lem:range}
\end{lem}

\begin{proof} 
This follows simply from the fact that 
$$ \frac{1}{2}\, \sup_{s,t\in[0,1]} |Y(t)-Y(s)| \leq \sup_{t\in[0,1]}|Y(t)|  
=  \sup_{t\in[0,1]} |Y(t)-Y(0)| \leq  \sup_{s,t\in[0,1]} |Y(t)-Y(s)|.
$$ 
and the de Bruijn Tauberian Theorem (Theorem~4.12.9 in \cite{bgt}).\end{proof}

Now we can prove Theorem~\ref{thm:main}.

\begin{proof}[ of Theorem~\ref{thm:main}] 
By assumption, for some constants $C_1,C_1', C_2,C_2'>0$ and {\it all} $\eps>0$, 
\begin{equation}
C_1' \, e^{- C_1 \eps^{-\theta}} \leq \pr{ \sup_{t\in[0,1]} |X(t)| \leq \eps} 
\leq C_2' \, e^{-C_2 \eps^{-\theta}}. 
\label{eqn:sballx}
\end{equation}
Let 
$$
N:=\inf_{t\in[0,1]} Y(t)\qquad\text{and}\qquad M:=\sup_{t\in[0,1]} Y(t).
$$
Note that, since $Y$ is continuous, 
$$Y([0,1]) = [N,M].$$ 
Therefore, by independence of $X$ and $Y$ and by independence of $X$ for positive and 
negative arguments, we have
\begin{multline}
\pr{ \sup_{t\in[0,1]} |X(Y(t))| \leq \eps} 
= \pr{ \sup_{s\in [N,0]} |X(s)| \leq \eps, \sup_{s\in[0,M]} |X(s)| \leq \eps} 
\label{eqn:crit1}
\\ 
= \E\left[  \pr{ \left. \sup_{s\in[N,0]} |X(s)| \leq \eps\right|Y} 
            \pr{ \left.\sup_{s\in[0,M]} |X(s)| \leq \eps \right| Y } 
     \right].
\end{multline}
Now we use the $H$-self-similarity of $X$ to see that the last expression equals
\begin{equation} 
  \E\left[  \pr{ \left. \sup_{s\in[0,1]} |(-N)^H X(s)| \leq \eps\right|Y} 
             \pr{ \left.\sup_{s\in[0,1]} |M^H X(s)| \leq \eps \right| Y } \right].
\label{eqn:condid}
\end{equation}
By (\ref{eqn:sballx}), we have 
$$ \pr{ \left. \sup_{s\in[0,1]} |(-N)^H X(s)| \leq \eps\right|Y} 
=  \pr{ \left. \sup_{s\in[0,1]} |X(s)| \leq \frac{\eps}{(-N)^H}\right|Y}
\leq C_2' e^{-C_2 \eps^{-\theta} (-N)^{H \theta}}.
$$
Analogously one can argue for the second term in (\ref{eqn:condid}), 
which yields that the whole expression in (\ref{eqn:condid}) is less than 
$$
C_2'^2 \E e^{-C_2 \eps^{-\theta} ( (-N)^{H \theta} + M^{H \theta})} 
\leq C_2'^2  \E e^{-\tilde{C}_2 \eps^{-\theta} ( M - N)^{H \theta}} 
= C_2'^2 \E e^{-\tilde{C}_2 \eps^{-\theta} (\sup_{s,t\in[0,1]} |Y(t)-Y(s)|)^{H \theta}}.
$$ 
By Lemma~\ref{lem:range}, the logarithmic order of this Laplace transform, when $\eps\to 0$, is 
$\eps^{-\theta/(1+H\theta/\tau)}$, which proves the upper bound in the assertion. 
The lower bound is established in exactly the same way using the lower bound in (\ref{eqn:sballx}).  
Note that this argument fails when $Y$ is not continuous, because we only have 
$Y([0,1]) \subsetneq [N,M]$.
\end{proof}

Now let us prove the strong asymptotics result.

\begin{proof}[ of Theorem~\ref{thm:mainstrong}]
Let $\delta>0$. By assumption, for all $0<\eps<\eps_0=\eps_0(\delta)$, 
$$
e^{- k (1+\delta) \eps^{-\theta}} \leq \pr{ \sup_{t\in[0,1]} |X(t)| \leq \eps} 
\leq e^{- k (1-\delta) \eps^{-\theta}}.
$$
This implies that there are constants $C_1, C_2>0$ (depending on $\eps_0$) 
such that for {\it all} $\eps>0$,
\begin{equation}
C_1 e^{- k (1+\delta) \eps^{-\theta}} \leq \pr{ \sup_{t\in[0,1]} |X(t)| \leq \eps} 
\leq C_2 e^{- k (1-\delta) \eps^{-\theta}}. 
\label{eqn:sballxstrong}
\end{equation}
By repeating the previous proof with (\ref{eqn:sballx}) replaced by (\ref{eqn:sballxstrong}),
we arrive at
\[
\pr{ \sup_{t\in[0,1]} |X(Y(t))| \leq \eps} \le
C_2^2 \E e^{-k(1-\delta) \eps^{-\theta} ( (-N)^{H \theta} + M^{H \theta})}.
\]
By using the assumption $H \theta=1$, we clearly have
$$
C_2^2 \E e^{-k(1-\delta) \eps^{-\theta} ( (-N)^{H \theta} + M^{H \theta})} 
= C_2^2  \E e^{-k(1-\delta) \eps^{-\theta} ( M - N)} 
= C_2^2 \E e^{-k(1-\delta) \eps^{-\theta} \sup_{s,t\in[0,1]} |Y(t)-Y(s)|}.
$$ 
Next, by the de Bruijn Tauberian Theorem 
(Theorem~4.12.9 in \cite{bgt}), the strong asymptotic logarithmic order of this Laplace transform, 
when $\eps\to 0$, is 
\[
\kappa^{1/(1+\tau)} \tau^{-1/(1+1/\tau)} (1+\tau) (k(1-\delta))^{1/(1+1/\tau)} 
\eps^{-\theta/(1+1/\tau)}.
\] 
Letting $\delta\to 0$ proves the upper bound in the assertion. 
The lower bound follows in exactly the same way using the lower bound in (\ref{eqn:sballxstrong}).  
As in the previous theorem, the proof fails when $Y$ is not continuous, because we only have 
$Y([0,1])\subsetneq  [N,M]$.
\end{proof}

\begin{proof}[ of Lemma~\ref{lem:anders}] Clearly, 
$$
\sup_{s,t} |Y(s)-Y(t)| \leq 2 \sup_t |Y(t)|
$$ 
and therefore 
$$
- \log \pr{ \sup_{s,t} |Y(s)-Y(t)| \leq 2\eps} \leq - \log \pr{ \sup_{t} |Y(t)| \leq \eps},
$$ 
which implies the inequality in one direction.

On the other hand, let $N:=\inf_{t} Y(t)$ and $M:=\sup_{t} Y(t)$. 
Fix $h>1$ and $0<\eps<1$.

Assume 
\begin{equation} 
\label{e:0418star} 
  \sup_{s,t} |Y(s)-Y(t)|\leq 2\eps. 
\end{equation} 
Then since $Y(t_0)=0$ we have $M\leq 2\eps$ and $N\geq -2\eps$.
Moreover, $Q:=\frac{M+N}2$ must satisfy $|Q|\leq \eps$.
 Furthermore, we have
$$
 \sup_{s} \left|Y(s)- Q\right|\leq \max\{M-Q, Q-N\}=\frac{M-N}2\le\eps.
$$

Let $m$ be the point in $\lbrace k\eps^h,k\in\Z\rbrace$ closest to $Q$.  There are only 
$\leq 2(\eps+\eps^h)/\eps^h$ possible values for $m$. Additionally, we have 
$$
\sup_s |Y(s)-m| = \sup_s \left|Y_s - Q + Q - m\right|\leq \eps + \eps^h.
$$

Therefore, 
\begin{eqnarray*} && \pr{ \sup_{t,s} |Y(s)-Y(t)| \leq 2\eps}
\\
&=&   \sum_{\lbrace k\in\Z,\, |k \eps^h|\leq \eps+\eps^h\rbrace}  
      \pr{ \sup_{t,s} |Y(s)-Y(t)| \leq 2\eps, m=k \eps^h}
\\
&\leq& \sum_{\lbrace k\in\Z,\, |k \eps^h|\leq \eps+\eps^h\rbrace}  
       \pr{ \sup_{s} |Y(s)-m| \leq \eps+\eps^h, m=k \eps^h}
\\
&\leq& \sum_{\lbrace k\in\Z,\, |k \eps^h|\leq \eps+\eps^h\rbrace}  
        \pr{ \sup_{s} |Y(s)-k \eps^h| \leq \eps+\eps^h}
\\
&\leq& \sum_{\lbrace k\in\Z,\, |k \eps^h|\leq \eps+\eps^h\rbrace}  
        \pr{ \sup_{s} |Y(s)| \leq \eps+\eps^h}
\\
&\leq& \frac{2(\eps+\eps^h)}{\eps^h}\,  \pr{ \sup_{s} |Y(s)| \leq \eps+\eps^h}
\\
&\leq & 4\eps^{1-h}  \pr{ \sup_{s} |Y(s)| \leq \eps+\eps^h},
\end{eqnarray*}
where we used the Anderson property in the fourth step.

Taking logarithms, multiplying with $-(2\eps)^\tau \ell(2\eps)^{-1}$, taking limits, 
and using that $\ell$ is a slowly varying function implies that 
\begin{eqnarray*} 
 && \lim_{\eps\to 0}  \eps^{\tau} \ell(\eps)^{-1}
   \left( - \log \pr{ \sup_{s,t} |Y(s)-Y(t)| \leq \eps} \right)
 \\
 &\geq& 2^\tau \lim_{\eps\to 0} \eps^\tau \ell(\eps)^{-1} 
   \left(-\log \pr{ \sup_s |Y(s)| \leq \eps + \eps^h} \right) 
 \\
 &=& 2^\tau \lim_{\eps\to 0} \left(\frac{\eps}{\eps+\eps^h}\right)^\tau (\eps+\eps^h)^\tau 
     \ell(\eps+\eps^h)^{-1}\left(-\log \pr{ \sup_s |Y(s)| \leq \eps + \eps^h} \right)
 \\
 &=& 2^\tau \lim_{\eps\to 0} \eps^\tau \ell(\eps)^{-1}
     \left(-\log \pr{ \sup_s |Y(s)| \leq \eps} \right),
 \end{eqnarray*} 
 which finishes the proof.

\end{proof}

\begin{proof}[ of Lemma~\ref{lem:anderson}] 
It is sufficient to check (\ref{anderson}) for cylindric sets. 
Fix $d\geq 1$ and let $B$ be a symmetric convex set in $\R^d$. 
Let $t_1, \ldots, t_d\in T$ and fix any function $e:T\to \R^1$. Define a cylinder 
\[
A:=\left\{ a:T\to \R:\, (a(t_1), \ldots, a(t_d))\in B\right\}
\]
and the corresponding random cylinders
\[
A_{Y,e}:=\left\{ f:\R\to \R:\, (f(Y(t_1))-e(t_1), \ldots, f(Y(t_d))-e(t_d)\in B\right\}.
\]
Then we have
\begin{eqnarray*}
 \pr{ X\circ Y \in A+e} 
 &=& \pr{(X(Y(t_1))-e(t_1), \ldots, X(Y(t_d))-e(t_d))\in B}
 \\
 &=&\E \, \pr{ X \in A_{Y,e}|Y} 
 =\E \, \pr{ X \in A_{Y,0} +(e(t_1), \ldots, e(t_d))|Y}
 \\
 &\le& \E \, \pr{ X \in A_{Y,0}|Y} 
= \pr{ X\circ Y \in A}.
\end{eqnarray*}

\end{proof}

\section{Examples} \label{sec:examples}
\subsection{Iterated Brownian motions}
As a first example let us consider $n$-iterated Brownian motions: 
\begin{equation} 
X^{(n)}(t):= X_n(X^{(n-1)}(t)),\qquad X^{(1)}(t)=X_1(t), 
\label{eqn:nbm} 
\end{equation} where the $X_i$ are independent (two-sided) Brownian motions. This process is $2^{-n}$-self-similar. 
The small deviation problem can be solved by using $(n-1)$-times 
Theorem~\ref{thm:mainstrong} (and Lemmas~\ref{lem:anders} and~\ref{lem:anderson}): 
$$
- \log \pr{ \sup_{t\in[0,1]} |X^{(n)}(t)| \leq \eps} 
\sim \frac{\pi^2}{8}\, d_n\, \eps^{-\tau_n},
$$
where $d_1:=1$ and 
\begin{equation}\label{eqn:dntaun}
d_n:=d_{n-1}^{1/(1+\tau_{n-1})} (2/\tau_{n-1})^{\tau_{n-1}/(1+\tau_{n-1})}(1+\tau_{n-1}),
   \qquad \tau_n=1/(1-2^{-n}).
\end{equation}
An explicit calculation yields:  $$(\tau_n)=\left(2,\frac{4}{3},\frac{8}{7},\frac{16}{15},\frac{32}{31},\ldots\right),$$
$$(d_n)=\left(1,3,\frac{7}{2^{4/7}},\frac{15}{2 \cdot 2^{1/3}},\frac{31}{4 \cdot 2^{6/31}},\ldots\right).$$
By induction one can prove that 
$$
d_n=\frac{2^n-1}{2^{n-3}\cdot 2^{(n+1)/(2^n-1)}}\, ,\qquad n\geq 1.
$$
Indeed, by using $1+\tau_{n-1}=\frac{2^n-1}{2^{n-1}-1}$, we obtain
\begin{eqnarray*}
d_n^{1+\tau_{n-1}} &=&
d_{n-1} (2/\tau_{n-1})^{\tau_{n-1}}(1+\tau_{n-1})^{1+\tau_{n-1}}
\\
&=& \frac{2^{n-1}-1}{2^{n-4}\cdot 2^{n/(2^{n-1}-1)}} \
   \left(\frac{2^{n-1}-1}{2^{n-2}}\right)^{\frac{2^{n-1}}{2^{n-1}-1}} \
   \frac {(2^n-1)^{1+\tau_{n-1}}}{(2^{n-1}-1)^{\frac{2^n-1}{2^{n-1}-1}}}
\\
&=&    (2^{n-1}-1)^{1+\frac{2^{n-1}}{2^{n-1}-1}- \frac{2^{n}-1}{2^{n-1}-1}} \
       (2^n-1)^{1+\tau_{n-1}} \
       2^{-\left[n-4+ \frac{n}{2^{n-1}-1}+(n-2)\frac{2^{n-1}}{2^{n-1}-1} \right]}
\\
&=&
(2^n-1)^{1+\tau_{n-1}} \
       2^{- \frac{n2^{n}-2^{n+1}-2^n+4}{2^{n-1}-1}}\ .
\end{eqnarray*}
Using the identity $n2^{n}-2^{n+1}-2^n+4=(n-3)(2^n-1)+n+1$, we arrive at
\[
d_n=(2^n-1) \
       2^{- \frac{(n-3)(2^n-1)+n+1}{2^{n}-1}}=(2^n-1) \
       2^{-(n-3)- \frac{n+1}{2^{n}-1}}\ ,
\]
as required in (\ref{eqn:dntaun}).

We summarize:
\begin{cor} 
Let $X^{(n)}$ be the process given by $(\ref{eqn:nbm})$, where the $X_i$ are 
independent two-sided Brownian motions. Then 
$$
- \log \pr{ \sup_{t\in[0,1]} |X^{(n)}(t)| \leq \eps} 
\sim \pi^2\, \frac{1-2^{-n}}{2^{(n+1)/(2^n-1)}} \ \eps^{-1/(1-2^{-n})}.
$$
\end{cor}

\begin{figure}[ht]\centering\includegraphics[scale=0.6]{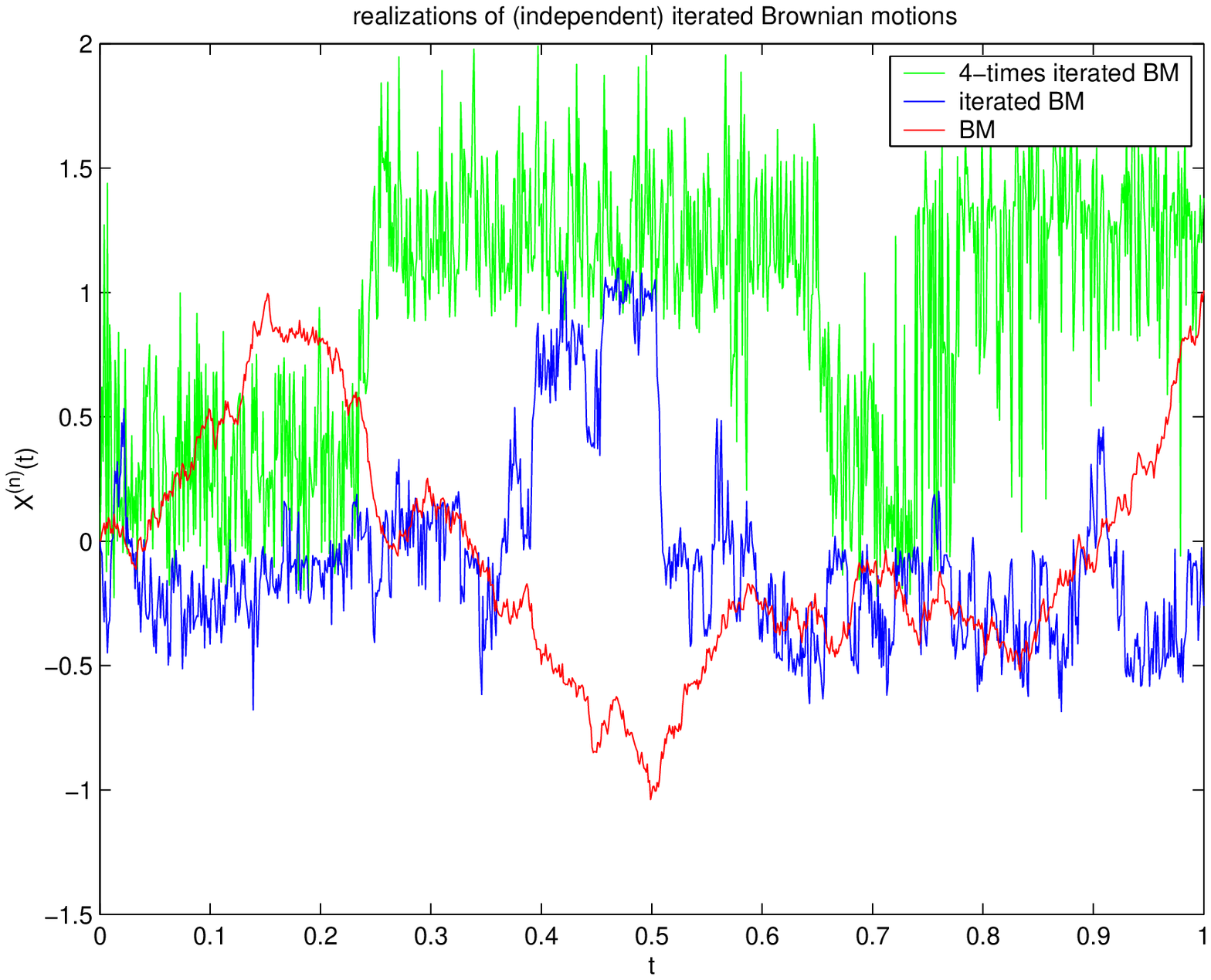}

Figure 1: typical sample paths of iterated Brownian motions\end{figure}

\subsection{Iterated two-sided fractional Brownian motions}
More generally, one can consider $n$-iterated fractional Brownian motions, given by 
(\ref{eqn:nbm}), where this time $X_1,\ldots, X_n$ are independent  (two-sided) fractional 
Brownian motions with Hurst parameters $H_1, \ldots, H_n$, respectively. The process 
$X^{(n)}$ is $H_1\cdot \ldots \cdot H_n$-self-similar. Its small deviation order is 
given by 
\begin{equation}
\label{eqn:h1hn}
- \log \pr{ \sup_{t\in[0,1]} |X^{(n)}(t)| \leq \eps} \sim c_n \eps^{-\tau_n},\qquad 
  \text{where $\frac{1}{\tau_n}=\sum_{j=1}^n H_j\cdot \ldots\cdot H_n$}
\end{equation} 
and $c_n$ is defined iteratively by 
$$ 
c_n:=(1+\tau_{n-1})\left[ c_{n-1}^{1/\tau_{n-1}}\,\frac{2 c(H_n)}{\tau_{n-1}}   
                   \right]^{\tau_{n-1}/(1+\tau_{n-1})}  ,
   \qquad c_1=c(H_1),
$$ 
and $c(H)$ is the small deviation constant of fractional Brownian motion with 
Hurst parameter $H$.

Even for $n=2$, i.e.\ fractional Brownian motions $X_1$ and $X_2$ with Hurst 
pa\-ra\-meters $H_1$, $H_2$, respectively, this leads to the new result that 
the small deviation order is 
$$
- \log \pr{ \sup_{t\in[0,1]} |X_2(X_1(t))| \leq \eps} 
\\ \sim \left(1 +\frac{1}{H_1}\right) 
\left[ c(H_1)^{H_1}\, 2 H_1 c(H_2)\right]^{1/(1+H_1)} \eps^{-\frac{1}{H_2(1+H_1)}}.
$$ 
In particular, for $H_1=H_2=:H$, i.e.\ iterated fractional Brownian motion, we get 
$$
- \log \pr{ \sup_{t\in[0,1]} |X_2(X_1(t))| \leq \eps} 
\sim c(H) \left( 2 H \right)^{1/(1+H)} \left(1 +\frac{1}{H}\right)  \eps^{-1/(H(1 + H))}.
$$

\subsection{The `true' iterated fractional Brownian motion}
Note that in the last subsection we obtained the small deviation order for `iterated 
fractional Brownian motion' $X\circ Y$, where $X$ was a {\it two-sided} fractional 
Brownian motion, i.e.\ a process consisting of two independent branches for positive 
and negative arguments, and $Y$ was another fractional Brownian motion (independent 
of the two branches of $X$). We shall now calculate the small deviation order for the 
`true' iterated fractional Brownian motion, namely, using $Y$ as above but $X$ being
a centered Gaussian process on $\R$ with covariance 
\begin{equation}
\E X(t)X(s)= \frac 12 \left(|s|^{2H}+|t|^{2H}-|t-s|^{2H}\right),\qquad t,s\in\R.
\label{eqn:truefbm}
\end{equation}

The general result is as follows.

\begin{thm}
Let $X$ be a fractional Brownian motion with Hurst index $H$ as given in $(\ref{eqn:truefbm})$ 
and $Y$ be a continuous process independent of $X$ satisfying 
\begin{equation}
-\log \pr{ \sup_{t\in[0,1]} |Y(t)-Y(s)|\leq \eps} \sim \kappa \eps^{-\tau}.
\label{e:tobechecked}
\end{equation} 
Then 
$$
-\log \pr{ \sup_{t\in[0,1]} |X(Y(t))|\leq \eps} 
\sim \kappa^{1/(1+\tau)} \tau^{-\tau/(1+\tau)} (1+\tau) c(H)^{\tau/(1+\tau)} 
\eps^{-\tau/(H(1+\tau))},
$$ 
where $c(H)$ is the small deviation constant of fractional Brownian motion. 
\label{thm:truefbm}
\end{thm}

This theorem can be applied to many processes $Y$. We recall that (\ref{e:tobechecked}) 
can be obtained e.g.\ via Lemma~\ref{lem:anders} from the small deviation order. 
In particular, if $Y$ is also a fractional Brownian motion we get the following result 
for the `true' iterated Brownian motion.

\begin{cor}  
Let $X$ be a fractional Brownian motion with Hurst index $H=H_2$ as given 
in $(\ref{eqn:truefbm})$ and $Y$ be a (continuous modification of a) fractional 
Brownian motion with Hurst index $H_1$ (independent of $X$). Then 
$$
-\log \pr{ \sup_{t\in[0,1]} |X(Y(t))|\leq \eps} 
\sim \left(1 +\frac{1}{H_1}\right) \left[ c(H_1)^{H_1}\, 2 H_1 c(H_2)\right]^{1/(1+H_1)} 
\eps^{-\frac{1}{H_2(1+H_1)}}.$$
\end{cor}

Recall that we obtain the same logarithmic small deviation order as for a two-sided fBM. Moreover,
the iteration of $n$ `true' fractional Brownian motions provides the same asymptotics as
obtained in (\ref{eqn:h1hn}) for the two-sided fBM.

Note that, in spite of the identity of the assertions, Theorem~\ref{thm:truefbm} does not follow from 
Theorem~\ref{thm:mainstrong}, since $X$ is not two-sided. We will show now how the stationarity of
increments of the `true' fBM replaces the independence property of the two-sided process. 
For the proof of Theorem~\ref{thm:truefbm}, we need the following lemma.

\begin{lem} \label{lem:truefbm}
For any $\delta\in (0,1)$ there exists $K_\delta>0 $ such that 
for all $N\le 0 \le M$, for all $\eps>0$, and for 
any centered Gaussian process $X(t), t\in \R$, with stationary increments it is true that
\begin{eqnarray*}
& & \pr{\sup_{0\le t \le M+|N|}|X(t)|\le (1-\delta)\eps} \pr{|X (N)|\le \eps/K_\delta}
\\
&\le&
\pr{\sup_{N\le t \le M}|X(t)|\le\eps}
\\
&\le&
\pr{\sup_{0\le t \le M+|N|}|X(t)|\le(1+\delta)\eps} \pr{|X (N)|\le \eps/K_\delta}^{-1}.
\end{eqnarray*}
\end{lem}

\begin{proof} 
To see the upper bound observe that the stationarity of increments and weak correlation 
inequality (cf.\ \cite{licorrelation}) yield
\begin{eqnarray*}
\pr{\sup_{0\le t \le M+|N|}|X(t)|\le(1+\delta)\eps} 
&=&
\pr{\sup_{N\le t \le M}|X(t)- X (N)|\le(1+\delta)\eps}
\\
&\ge&  \pr{\sup_{N\le t \le M}|X(t)|\le(1+\delta/2)\eps,
|X (N)|\le \delta\eps/2 }
\\
&\ge&
\pr{\sup_{N\le t \le M}|X(t)|\le \eps}
\pr{|X (N)|\le \eps/K_\delta }.
\end{eqnarray*}

For the lower bound, using the same arguments in inverse order we get
\begin{eqnarray*}
\pr{\sup_{N\le t \le M}|X(t)|\le\eps} 
&=& \pr{\sup_{N\le t \le M}|X(t)-X (N)+X (N)|\le\eps}\\
&\ge&
\pr{\sup_{N\le t \le M}|X(t)-X (N)|\le(1-\delta/2)\eps; |X (N)|\le \delta\eps/2}
\\
&\ge&
\pr{\sup_{N\le t \le M}|X(t)-X (N)|\le(1-\delta)\eps }
\pr{|X (N)|\le \eps/K_\delta }
\\
&=&
\pr{\sup_{0\le t \le |N|+M}|X(t)|\le(1-\delta)\eps}
\pr{|X (N)|\le \eps/K_\delta }.
\end{eqnarray*}\end{proof}


\begin{proof}[ of Theorem~\ref{thm:truefbm}] 
Let $\delta>0$ and define as before 
$N:=\inf_{t\in [0,1]} Y(t)$ and $M:=\sup_{t\in [0,1]} Y(t)$. 
Then Lemma~\ref{lem:truefbm} yields that, for some constant $K_\delta>0$,
\begin{eqnarray*}
 && \pr{\sup_{t\in[0,1]}|X(Y(t))|\le\eps}
 \\
 &=& \E\,\left[ \pr{\left.\sup_{N\le t \le M}|X(t)|\le\eps\right| Y} \right] 
 \\
 &\geq& \E \,\left[ \pr{\left. \sup_{0\le t \le M+|N|}|X(t)|\le (1-\delta)\eps \right| Y} 
   \pr{\left.|X (N)|\le \frac{\eps}{K_\delta}\right|Y}\right]
 \\
&=& \E \, \left[\pr{\left. \sup_{0\le t \le 1}|X(t)|\le \frac{(1-\delta)\eps}{(M+|N|)^{H}} \right| Y} 
   \pr{\left.|X (1)|\le \frac{\eps}{K_\delta |N|^{H}}\right|Y}\right]
\\
&\geq & \E \,\left[ \pr{\left. \sup_{0\le t \le 1}|X(t)|\le \frac{(1-\delta)\eps}{(M+|N|)^{H}} \right| Y} 
    \pr{\left.|X (1)|\le \frac{\eps}{K_\delta (M+|N|)^{H}}\right|Y}\right]
\\
&=:&\E \left[ f(M+|N|) g(M+|N|) \right].
\end{eqnarray*}
Note that $f,g\geq 0$ are non-increasing functions. Thus, by the FKG inequality (cf.\ e.g.\ \cite{liggett}, p.\ 65), 
the last term is bounded from below by
\begin{multline*} 
\E f(M+|N|) \cdot \E g(M+|N|) 
\\ 
= \E \left[\pr{\left. \sup_{0\le t \le 1}|X(t)|\le \frac{(1-\delta)\eps}{(M+|N|)^{H}} \right| Y}\right] 
   \cdot \E \, \pr{|X (1)|\le \frac{\eps}{K_\delta (M+|N|)^{H}}}.
\end{multline*}
The first term can be handled as in the proof of Theorem~\ref{thm:mainstrong}, 
the resulting order is 
\begin{multline*} 
- \log \E\, \pr{\left. \sup_{0\le t \le 1}|X(t)|\le \frac{(1-\delta)\eps}{(M+|N|)^{H}}  \right| Y} 
\\ 
\sim \kappa^{1/(1+\tau)} \tau^{-\tau/(1+\tau)} (1+\tau) c(H)^{\tau/(1+\tau)} 
   ((1-\delta)\eps)^{-\tau/(H(1+\tau))}.
\end{multline*}
On the other hand, one easily proves that 
$$
\E \, \pr{|X (1)|\le \frac{\eps}{K_\delta (M+|N|)^{H}}}
$$ 
admits a lower bound of order $\eps$, as $\eps\to 0$. 
Therefore, 
\begin{multline*}
-\log \pr{\sup_{t\in[0,1]}|X(Y(t))|\le\eps} 
\\ \lesssim 
\kappa^{1/(1+\tau)} \tau^{-\tau/(1+\tau)} (1+\tau) c(H)^{\tau/(1+\tau)} 
((1-\delta)\eps)^{-\tau/(H(1+\tau))}.
\end{multline*}
Letting $\delta\to 0$ finishes the proof. The upper bound can be proved along the same lines
or by using H\"older inequality instead of FKG inequality.

\end{proof}

\section{The example of \texorpdfstring{$\alpha$}{alpha}-time Brownian motion} \label{sec:alphatime}
\subsection{Motivation}
Let $X$ be a Brownian motion and $Y$ be a symmetric $\alpha$-stable L\'{e}vy process.

In \cite{nane}, the small deviation problem for $X\circ Y$, called $\alpha$-time 
Brownian motion there, is studied and further applied to the LIL of Chung type and
 results for the local time for that processes. However, the method of proof in \cite{nane} is 
 essentially the same as for our Theorem~\ref{thm:mainstrong}. Note that this proof 
 is wrong in the case of $\alpha$-time Brownian motion, since the inner process $Y$ 
 is {\it not continuous}, which is a main ingredient of the proof. In fact, it is used 
 that $Y([0,1]) = [N,M]$, with as above $N:=\inf_t Y(t)$ and $M:=\sup_t Y(t)$, 
 which is not true for this $Y$.

However, trivially $Y([0,1]) \subsetneq [N,M]$, and thus the proof in \cite{nane} 
does give a {\it lower bound} for the small deviation probability. The purpose of this 
section is to give a correct proof of the upper bound.

More precisely, we show the following version of Theorem~2.3 in \cite{nane}. 
This result implies weaker versions of the results in \cite{nane}.

\begin{thm} 
Let $X$ be a two-sided Brownian motion and $Y$ be a strictly $\alpha$-stable L\'{e}vy process 
(independent of $X$) that is not a subordinator. Then 
$$ 
-\log \pr{ \sup_{t\in[0,1]} |X(Y(t))| \leq \eps } 
\approx \eps^{-2\alpha/(1+\alpha)}.$$ 
\label{thm:alphatime}
\end{thm}

We note that this is not as strong as the assertion of Theorem~2.3 in \cite{nane}: 
the existence of the small deviation constant and its value are not assured. 
This should be subject to further investigation.

Note furthermore that we prove the result for general strictly stable L\'{e}vy processes 
$Y$, symmetry is not a feature that would be required here. The only property that 
is used is self-similarity.

For the sake of completeness let us mention that in the case that $Y$ is an $\alpha$-stable 
subordinator ($0<\alpha<1$), the above result is wrong. Namely, in that case $X\circ Y$ is 
in fact a \emph{symmetric} $(2\alpha)$-stable L\'{e}vy process itself, so that we then get 
that 
$$
-\log \pr{ \sup_{t\in[0,1]} |X(Y(t))| \leq \eps } \approx \eps^{-2\alpha}.
$$

We shall even prove the following more general version of Theorem~\ref{thm:alphatime}.

\begin{thm} 
Let $X$ be a two-sided strictly $\beta$-stable L\'{e}vy process ($0< \beta\leq 2$) 
and $Y$ be a strictly $\alpha$-stable L\'{e}vy process ($0< \alpha\leq 2$, independent of $X$) 
that is not a subordinator. Then 
$$
-\log \pr{ \sup_{t\in[0,1]} |X(Y(t))| \leq \eps } \approx \eps^{-\beta \alpha/(1+\alpha)}.
$$ 
\label{thm:betatime}
\end{thm}

\begin{rem}
The result is also true if we take an $H$-fractional Brownian motion or $H$-Riemann-Liouville process 
as $X$, cf.\ \cite{lindezipfel}. Then of course, $\beta=1/H$.
\end{rem}

The proof of Theorem~\ref{thm:betatime} is given in several steps. First note that the lower 
bound follows from our Theorem~\ref{thm:main} and Prop.~3, Section~VIII, 
in \cite{bertoin} (the result actually dates back to \cite{taylor}, 
\cite{mogulskii}, and \cite{bormog}). The upper bound follows from 
Proposition~\ref{prop:coverlp} below, as explained there.

\subsection{Handling the outer Brownian motion}
In order to prove Theorem~\ref{thm:betatime}, we shall proceed as follows. In a first step, 
we show that the small deviation problem of processes that are subordinated to Brownian 
motion (or more generally, to a strictly $\beta$-stable L\'{e}vy process) are closely 
connected to the (random) entropy numbers of the range of the inner process 
(i.e.\ $K=Y([0,1])$). This technique was previously used in \cite{lindeshi} and 
\cite{lindezipfel} for fractional Brownian motion. Then we estimate the entropy numbers 
of the range of the inner process, in our case a strictly $\alpha$-stable L\'{e}vy process 
(the subordinator case was studied in \cite{lindeshi}). This requires completely new arguments.

To formulate the first step, let us define the following notation. For given $\eps>0$ and a 
compact set $K\subseteq \R$, let 
$$
N(K,\eps):=\min\left\lbrace n\geq 1, \exists x_1,\ldots, x_n \in \R : 
\forall x\in K \exists i\leq n :\, |x-x_i|\leq \eps \right\rbrace.
$$ 
These quantities are usually called covering numbers of the set $K$ and characterize its metric 
entropy. We can get rid of the randomness of the 
outer process $X$ in the following way.

\begin{prop} 
Let $X$ be a (two-sided) strictly $\beta$-stable L\'{e}vy process, $0< \beta\leq 2$. 
Then there is a constant $c_0>0$ such that, for all compact sets $K\subset \R$ and for 
all $\eps>0$,  
$$
 \pr{ \sup_{t\in K} |X(t)| \leq \eps } \leq e^{1-N(K,c_0 \eps^\beta)}.
$$
\end{prop}
\begin{proof}
This simple result can be proved in essentially the same way as Proposition~3.1 in \cite{lindezipfel}, 
where it was shown for $X$ being fractional Brownian motion. We therefore 
just indicate the proof: choose $c_0$ so large that $\sup_{t\ge c_0} \pr{|X(t)|\le 2}\le e^{-1}$.
For $N=N(K,c_0 \eps^\beta)$ find an increasing sequence $t_1,...,t_N$ in $K$ such that
$t_{i+1}-t_i\ge c_0 \eps^\beta$ for all $i=1,...,N-1$. Then by independence of increments and strict
stability of $X$ we have
\begin{eqnarray*}
 \pr{ \sup_{t\in K} |X(t)| \leq \eps } 
 &\leq& 
\pr{ \sup_{1\le i\le N} |X(t_i)| \leq \eps } 
\leq 
\pr{ \sup_{1\le i\le N-1} |X(t_{i+1})-X(t_i)| \leq 2\eps }
\\
&=&
\prod_{i=1}^{N-1} \pr{  |X(t_{i+1})-X(t_i)| \leq 2\eps }
=\prod_{i=1}^{N-1} \pr{  |X(t_{i+1}-t_i)| \leq 2\eps }
\\
&=&
\prod_{i=1}^{N-1} \pr{  |X\left(\frac{t_{i+1}-t_i}{\eps^\beta}\right)| \leq 2 \eps }
\le e^{-(N-1)}.
\end{eqnarray*}
\end{proof}

Recall that in order to prove Theorem~\ref{thm:betatime} we want to get an upper bound for 
$$
\pr{ \sup_{t\in[0,1]} |X(Y(t))| \leq \eps } 
= \E \left[ \pr{ \left.\sup_{t\in K} |X(t)| \leq \eps \right| Y } \right] 
\leq \E \left[  e^{1-N(K,c_0 \eps^\beta)}\right],
$$
where we let $K=Y([0,1])$.
Since for any $R>0$ we have
\[
\E \left[  e^{-N(K,c_0 \eps^\beta)}\right] \le e^{-R} + \pr{N(K,c_0 \eps^\beta)\le R},
\]
the upper bound in Theorem~\ref{thm:betatime} follows immediately from the next result.

\begin{prop} 
Let $Y$ be a strictly $\alpha$-stable L\'{e}vy process and set $K=Y([0,1])$. Then
there exist small $c$ and $\delta$ depending on the law of $Y$ such that 
\begin{equation}
\pr{ N(K,\eps) < \delta k } \leq e^{-c k}, \label{main}
\end{equation}  
for  all $\eps>0$ and $k=\left\lceil\eps^{-\alpha/(1+\alpha)}\right\rceil$.
\label{prop:coverlp}
\end{prop}

The proof of Proposition~\ref{prop:coverlp} is given in the next subsection. 
In Section~\ref{sec:localtimeproof}, an alternative proof is given, which is much shorter 
but involves local times and thus only works for $\alpha>1$.

\begin{rem} Actually, the investigation of the small deviation probabilities for covering numbers such as  
$\pr{ N(Y([0,1]),\eps) < k }$
is an interesting problem in its own right and we hope to handle it elsewhere extensively. Here, we just
notice that the order of the estimate (\ref{main}) is sharp and that it is a particular case of a more general fact
that can be obtained similarly:
\[ - \log \pr{ N(Y([0,1]),\eps) < k } \approx  (k\eps)^{-\alpha},
\] 
which is valid for $1\le k \le \eps^{-1}$, $1<\alpha<2$, and for 
$1\le k\le \delta \eps^{-\frac{\alpha}{1+\alpha}}$, $0<\alpha<1$. 
More efforts are needed to understand the remaining
cases, e.g.\ $\eps^{-\frac{\alpha}{1+\alpha}}\ll k \ll \eps^{-\alpha}$, $0<\alpha<1$.
\end{rem}

\subsection{Proof of Proposition~\ref{prop:coverlp}} 
We will now prove inequality (\ref{main}). For this purpose, let us introduce the notation 
$$
N_{[0,t]}(\eps):=N(Y([0,t]),\eps), \qquad t\geq 0.
$$ 
For a given $t\geq 0$, $N_{[0,t]}(\eps)$ counts how many intervals are needed in order to cover 
the range of the process when only looking at the path until time $t$.

Let $T=\eps^{-\alpha}$. By scaling we have  
\[
  \pr{N_{[0,1]}(\eps)\le \delta k}=  \pr{ N_{[0,T]}(T^{1/\alpha}\eps)\le \delta k}
  =\pr{ N_{[0,T]}(1)\le \delta k}.
\]
By splitting the time interval $[0,T]$ in $k$ equal pieces, we get intervals of length
$L=T/k \ge k^\alpha$, since $T\ge k^{1+\alpha}$.

Observe that if $N_{[0,T]}(1)\le \delta k$, then there are at most $\lceil \delta k\rceil$ points 
where the function $t\mapsto N_{[0,t]}(\eps)$ increases. Thus, there are at least 
$\lfloor (1-\delta)k\rfloor$ of the intervals, where this function does not increase. 
Thus, there exists a set of integers $J\subseteq\{ 0,\ldots,k-1\}$ such that 
$|J|\geq\lfloor(1-\delta)k\rfloor$ 
and there is no increase of covering numbers
\begin{equation} \label{Nstops}
           N_{[0,(j+1)L]}(1)= N_{[0,jL]}(1), \qquad \forall j\in J.
\end{equation}
Let $\delta<1/2$. Notice that the number of choices for $J$ satisfying 
$|J|\geq \lfloor(1-\delta)k \rfloor$ can be expressed as $2^k \pr{B_k\geq \lfloor(1-\delta)k \rfloor}$
where $B_k$ is a sum of $k$ Bernoulli random variables attaining the values $0$ and $1$ with equal probabilities.
By the classical Chernoff bound for the large deviations of $B_k$ we see that this number is smaller than:
\begin{equation} \label {numberofJ}
 \left(\delta^{-\delta}(1-\delta)^{-(1-\delta)}\right)^k := \exp(\delta_1 k), 
\end{equation}
while $\delta_1$ satisfies $\delta_1\to 0$, as $\delta\to 0$.

For a while, we fix an index set $J$.
We enlarge the events from (\ref{Nstops}) as follows:
\begin{eqnarray*}
\Omega_j &:=& \left\{  N_{[0,(j+1)L]}(1)= N_{[0,jL]}(1) \le k  \right\}
\\
&\subseteq & \left\{ Y((j+1)L)\in Y[0,jL]+[-1,1],   N_{[0,jL]}(1) \le k \right\}
\\
&=& \left\{ Y((j+1)L)-Y(jL)\in Y[0,jL]+[-1,1] -Y(jL), N_{[0,jL]}(1) \le k   
\right\}
\\
&=:& \Omega'_j .
\end{eqnarray*}
Let, as usual, ${\cal F}_{t}$ denote the filtration generated by the process $Y$ up 
to time $t$. We have, by stationarity and independence of increments,
\begin{eqnarray*}
\esssup\ \P(\Omega'_j|{\cal F}_{jL}) &\le& 
\sup_A\left\{ \P(Y(L)\in A+[-1,1]),\, N(A,1)\le k \right\}
\\
&\le&
\sup_{A'} \left\{ \P(Y(L)\in A'),\, N(A',2)\le k \right\}
\\
&\le&
\sup_{A'} \left\{ \P(Y(L)\in A'),\, |A'| \le 2k \right\}
\\
&\le&
\sup_{A''} \left\{ \P(Y(1)\in A''),\, |A''| \le 2 \right\}=: c_1<1.
\end{eqnarray*}
By a standard conditioning argument, we find
\[
\P\left(\bigcap_{j\in J} \Omega_j\right) \le 
\P\left(\bigcap_{j\in J} \Omega'_j\right) \le
\prod_{j\in J} \esssup\ \P\left(\Omega'_j|{\cal F}_{jL}   \right)\le c_1^{|J|} \le c_1^{(1-\delta)k}.
\]
By summing up over all sets $J$, we have
\begin{eqnarray*}
\P(N_{[0,1]}(\eps^2)<\delta k) &\le& 
  \sum_{J:|J|\geq\lceil(1-\delta)k\rceil} \P\left(\bigcap_{j\in J} \Omega_j\right)
\\ &\le& |\{J:|J|\geq \lceil (1-\delta)k\rceil \}|  \ c_1^{(1-\delta)k}
\\
&\le& \exp(\delta_1 k) c_1^{(1-\delta)k}\, ;
\end{eqnarray*}
and we are done with (\ref{main}) if $\delta$ is chosen so small that
$\exp(\delta_1) c_1^{1-\delta}<1$.

\subsection{Alternative proof via local times} \label{sec:localtimeproof}
Here, we give an alternative proof for Proposition~\ref{prop:coverlp} when $\alpha>1$. 
In this case, the strictly $\alpha$-stable process $Y$ possesses a continuous local 
time $L$: 
$$ 
\int_B L(x)\, \d x = \int_0^1 \ind_{B}(Y(t))\, \d t,\qquad \text{for all Borel sets $B$.}
$$

We define $L^*=\sup_x L(x)$, the maximum of local time of the $\alpha$-stable L\'evy 
process considered on the time interval $[0,1]$. It is shown by Lacey (\cite{laceylocal}) that
\[
    \log \P(L^*> u) \sim -c u^\alpha, \qquad \text{as $u\to\infty$},
\]
for some (explicitly known) constant $c>0$. Therefore, for $\eps$ small enough,
\[
\P(N(K,\eps)<k)
\le
\P(L^* > (\eps k)^{-1})
\le \exp(- c'(\eps k)^{-\alpha}).
\]
In particular, by letting 
$k=\left\lceil \eps^{-\alpha/(1+\alpha)}\right\rceil$ we get
\[
\P(N(K,\eps)<k)
\le 
\exp(-c'' \eps^{-\alpha/(1+\alpha)}),
\]
as required in (\ref{main}).

\section{Some further remarks on extensions} \label{sec:remarks}

\paragraph{1) Slowly varying terms in the asymptotics of $Y$.} One may add a slowly varying term 
in the asymptotics in (\ref{e:weak:y}) or (\ref{e:knownrange}) and still use the same method 
of proof to the get a similar result. Also, one can consider what happens if $Y$ has polynomial 
small deviation behaviour or super-exponential behaviour and obtain a similar result using 
the same method of proof.  

However, it is not immediately clear what happens if $X$ has a small deviation order below or 
above the exponential scale.

\medskip
\paragraph{2) Chung's Law of the iterated logarithm.} It is straightforward to derive from the results 
from Theorem~\ref{thm:main} and Theorem~\ref{thm:mainstrong}  the lower bounds in the 
respective Chung law if, additionally, 
$Y$ satisfies a certain self-similarity. Notice that the derivation of the upper bounds usually needs some 
independence arguments that are specific for the considered processes. We preferred not to go into these details in order
to maintain a certain level of generality.

\medskip
\paragraph{3) Small deviations in $L_p$-norm.} Note that we investigate the small deviation problem 
for $X\circ Y$ w.r.t.\ the supremum norm in this article. It would be interesting to study the 
small deviations in other norms, e.g. for
$$
 \int_0^1 |X(Y(t))|^p \, \d t
$$
with given $p>0$, and investigate their dependence on the small deviations of $X$, $Y$ and, if applicable, 
the local time of $Y$. Nothing seems to be known about this problem even for iterated Brownian motion.

\hypertarget{vd}{~}\pdfbookmark[1]{References}{vd}


\def\cprime{$'$}
\providecommand{\bysame}{\leavevmode\hbox to3em{\hrulefill}\thinspace}
\providecommand{\MR}{\relax\ifhmode\unskip\space\fi MR }
\providecommand{\MRhref}[2]{%
  \href{http://www.ams.org/mathscinet-getitem?mr=#1}{#2}
}
\providecommand{\href}[2]{#2}

\end{document}